\documentclass[a4paper, oneside,11 pt]{article}

\usepackage[latin1]{inputenc}
\usepackage[T1]{fontenc}
\usepackage[english]{babel}

\usepackage[all]{xy}
\usepackage{verbatim}
\usepackage{amsmath}
\usepackage{amsthm}
\usepackage{graphicx}
\usepackage{amssymb}
\usepackage{epstopdf}
\usepackage{url}
\usepackage{amsfonts}
\usepackage{todonotes}
\newcommand{\mc}{\mathcal}
\newcommand{\pt}{\partial}

\newcommand{\br}{\mathbb{R}}

\newcommand{\te}{\theta}

\newcommand{\e}{\varepsilon}

\renewcommand{\r}{\rho}

\renewcommand{\(}{\left(}
\renewcommand{\)}{\right)}
\renewcommand{\[}{\left[}
\renewcommand{\]}{\right]}

\newtheorem{thm}{Theorem}

\newtheorem{prop}[thm]{Proposition}

\def\be{\begin{equation}}
\def\ee{\end{equation}}
\def\bea{\begin{eqnarray}}
\def\eea{\end{eqnarray}}

\newcommand{\Div}{\operatorname{div}}

\newcommand{\Hess}{\operatorname{Hess}}

\def \La {\big\langle}
\def \Ra {\big\rangle}
\numberwithin{thm}{section}
\numberwithin{equation}{section}

\title{Asymptotic analysis for a very fast diffusion equation arising from the 1D quantization problem}

\begin{document}

\author{
Mikaela Iacobelli
  \thanks{University of Cambridge, DPMMS Centre for Mathematical Sciences, Wilberforce road, Cambridge CB3 0WB, UK. Email: \textsf{iacobelli@maths.cam.ac.uk}}
%   \and
% Francesco Saverio Patacchini
% \thanks{Department of Mathematics, Imperial College London, South Kensington Campus, London SW7 2AZ, UK. Email: \textsf{f.patacchini13@imperial.ac.uk}}
%    }
}

\maketitle

\begin{abstract}
In this paper we study the asymptotic behavior of a very fast diffusion PDE in 1D with periodic boundary conditions.
This equation is motivated by the gradient flow approach to the problem of quantization of measures introduced in \cite{CGI1}.
We prove exponential convergence to equilibrium under minimal assumptions on the data, and we also provide sufficient conditions
for $W_2$-stability of solutions.
\end{abstract}

\section{Introduction}

During the last years, asymptotic analysis for solutions of nonlinear parabolic equations have attracted a lot of attention, also in connection with gradient flows and entropy methods.

The aim of the present paper is to investigate the dynamics of the PDE
\be\label{eq:vfd}
\pt_tf(t,x)=-r\pt_x\bigg(f(t,x)\pt_x\Big(\frac{\r(x)}{f(t,x)^{r+1}}\Big) \bigg)\qquad \text{in }(0,\infty)\times [0,1],
\ee
where $r>1,$ and $\rho>0$ and $f(t, \cdot)$ are probability densities on $[0,1]$ with periodic boundary conditions. 
When $\rho=1$, this equation takes the form
\be\label{eq:vfd1}
\pt_tf=-(r+1)\pt_x^2\big(f^{-r}\big),
\ee
which belongs to the general class of fast diffusion equations
$$
\pt_t u={\rm div}(u^{m-1}\nabla u),\qquad m<1.
$$
We recall that, when the problem is set on the whole space $\br^n$, the value of $m$ plays a crucial role: solutions are smooth and exist for all times if $m>m_c:=(n-2)/2$,
while they vanish in finite time if $m \leq m_c$ (the existence of such an extinction time motivates the name ``very fast diffusion equations'').
There is a huge literature on the subject, and we refer the interested reader to the monograph \cite{Va2} for a comprehensive overview and more references.

Our case corresponds to the range $m=-r<-1$. 
It is interesting to point out that \eqref{eq:vfd1} set on the whole space $\mathbb R$ or with zero Dirichlet boundary conditions has no solutions, since all the mass instantaneously disappear \cite[Theorem 3.1]{Vazquez} (see also \cite{E,RV,Vazquez_1} for related results).
It is therefore crucial that in our setting the equation has periodic boundary conditions,
so that the mass is preserved

We observe that this kind of equations has the property of diffusing extremely fast. In particular, if $f_0$ is a non-negative and bounded initial datum,
the solution becomes instantaneously positive. As we are only interested in the long time behaviour of solutions, to simplify the presentation
we will only consider initial data that are bounded away from zero and infinity.\\

Our equation \eqref{eq:vfd} is motivated by the so-called \emph{quantization problem}.
The term \emph{quantization} refers to the process of finding the \emph{best} approximation of a $d$-dimensional probability distribution by a convex combination of a finite number $N$ of Dirac masses.
%Much of the early attention in the engineering and statistical literature was concentrated on the one-dimensional quantization problem. 
This problem arises in several contexts and has applications in information theory (signal compression), numerical integration, and mathematical models in economics (optimal location of service centers). 
In order to explain the meaning of the equation  \eqref{eq:vfd}, we now briefly recall the gradient flow approach to the quantization problem
introduced in \cite{CGI1}, and further investigated in \cite{CGI2}.

Given $r\ge 1$, consider $\mu=\rho(x)\,dx$ a probability measure on an open set $\Omega \subset \mathbb{R}^n$.
   Given $N$ points $x^{1}, \ldots, x^{N} \in \Omega,$ one wants to find the  best approximation of $\mu,$ in the Wasserstein distance $W_r$, by a convex combination of Dirac masses centered at $x^{1}, \ldots, x^{N}.$ Hence one minimizes
$$
\inf \bigg\{ W_r\bigg(\sum_im_i \delta_{x^i}, \mu\bigg)^r\,:\, m_1, \ldots, m_N\ge0, \  \sum_im_i=1
\bigg\},
$$
with
$$
W_r(\nu_1, \nu_2):=\inf\bigg\{
\biggl(\int_{\Omega\times\Omega}|x-y|^rd\gamma(x,y)\biggr)^{1/r}\,:\,
(\pi_1)_\#\gamma=\nu_1,\  (\pi_2)_\#\gamma=\nu_2\bigg\},
$$
where $\gamma$ varies among all probability measures on $\Omega \times \Omega$, and $\pi_i: \Omega \times \Omega \to \Omega$ ($i=1,2$) denotes the canonical projection onto the $i$-th factor. See \cite{A, V} for more details on Wasserstein distances.

As explained in \cite[Chapter 1, Lemmas 3.1 and 3.4]{GL}, this problem is equivalent to minimizing the functional
 $$
F_{N,r} (x^{1}, \ldots, x^{N}) := \int_\Omega \underset{1\le i \le N}{\mbox{min}} | x^i-y |^r\,d\mu(y).
$$
To find a minimizer to this function, in \cite{CGI1} the authors introduce a \emph{dynamical approach}
where they study the dynamics of the gradient flow induced by $F_{N,r}$.
Since the main goal is to understand the structure of minimizers in the limit as $N$ tends to infinity,
in \cite[Introduction]{CGI1} and in \cite[Sections 2 and 3]{CGI2} the authors are able to find a formula for the ``limit'' of $F_{N,r}$ when $N\to \infty$.

As shown in \cite{CGI1}, when $n=1$ this limit is given by the functional 
$$
\mathcal F[X]:=\int_0^1 |\partial_\theta X|^{r+1}\,\rho(X)\,d\theta,
$$
and its $L^2$-gradient flow 
is given by the following non-linear parabolic equation
\begin{equation}\label{gradient flow}
\partial_tX=(r+1)\partial_\te\big(\rho(X)|\partial_\te X|^{r-1}\partial_\te X\big)
-\rho'(X)|\partial_\te X|^{r+1},
%\frac{1}{2^r(r+1)}
\end{equation}
coupled with the Dirichlet boundary condition.
%\begin{equation}
%\label{eq:boundary}
%X(t,0)=0, \qquad X(t,1)=1.
%\end{equation}
%
%Let us notice that, in the particular case $h \equiv 1,$  \eqref{gradient flow} becomes a {\bf $p$-Laplacian equation}
%$$
%\partial_tX=\frac{1}{2^r} \partial_\te\big(|\partial_\te X|^{r-1}\partial_\te X\big),\qquad p-1=r.
%$$
% 
This equation provides a Lagrangian description of the evolution of our system of particles in the limit $N \to \infty.$  
We can also study the Eulerian picture for \eqref{gradient flow}.
Indeed, if we denote by $f(t,x)$ the image of the Lebesgue measure through the map $X,$ i.e.,
$$
f(t,x)dx=X(t,\te)_\#d\te,
$$
then the PDE satisfied by $f$ takes the form (see \cite{A})
$$
\pt_tf(t,x)=-r\pt_x\bigg(f(t,x)\pt_x\Big(\frac{\r(x)}{f(t,x)^{r+1}}\Big) \bigg)
$$
with periodic boundary conditions,
and in view of the results in \cite{GL,CGI1}
we expect the following long time behavior:
$$
f(t,x) \longrightarrow \gamma \,\rho^{1/(r+1)}(x)\quad \text{as $t\to \infty$},\qquad \text{where }\gamma:=\frac{1}{\int_0^1 \r(y)^{1/(r+1)}dy}.
$$

More precisely, the results in \cite{CGI1} show the validity of the limit only when $r=2$
and under the assumption that $\rho$ is close to $1$ in $C^2$.
The goal here is to generalize and improve this result.
\\

Our starting point for studying the asymptotic behavior of \eqref{eq:vfd} is the observation that this equation can be seen as the gradient flow of the functional
\be\label{eq:F}
\mathcal F_\rho[f]:=\int_0^1\frac{\rho(x)}{f(x)^{r}} \,dx.
\ee
with respect to the $W_2$ distance.

In a first step, by exploiting a modulated $L^2$ energy method, we obtain exponential convergence to equilibrium under minimal assumptions on the density $\rho$.
Then, we investigate the displacement convexity of the functional $F_\rho$.
Notice that, as we shall prove in Proposition \ref{prop:Euler} below,
if $\rho$ and $f(0)$ are bounded away from zero, then $f(t)$ remains uniformly away from zero for all $t \geq 0$.
In particular \eqref{eq:vfd} is uniformly parabolic, and $f(t)$
is smooth if $\rho$ is so.

Since our focus is on the asymptotic behavior, we shall assume that $\rho$ is of class $C^{2,\alpha}$ for some $\alpha \in (0,1)$,
so that parabolic regularity theory ensures that $f(t)$ is of class $C^{2,\alpha}$ for all times, hence $f$ is a classical solution.
However our results are independent of the smoothness of $\rho$ and can be thought as a priori estimates.
In particular, we believe one could extend them to the setting of weak solutions by using the general theory of minimizing movements in \cite{AGS} (See also \cite{CS}).
Since our main goal is to understand the general asymptotic properties of the equation \eqref{eq:vfd},
we shall not investigate this here.

Our first result is the following:

\begin{thm}
\label{thm:1}
Let $\lambda \in (0,1]$, and assume that $\rho:[0,1]\to [\lambda,1/\lambda]$ is periodic and of class $C^{2,\alpha}$ for some $\alpha \in (0,1)$.
Let $f(0,\cdot):[0,1]\to \mathbb R$ satisfy $0<a_1 \leq f(0,\cdot) \leq A_1$, and let $f$ solve \eqref{eq:vfd} with periodic boundary conditions.
Then
\begin{equation}
\label{eq:bound f}
a:=a_1\lambda^{2/(r+1)}\leq f(t,x)\leq \frac{A_1}{\lambda^{2/(r+1)}}=:A\qquad \text{for all}\  \,t \geq 0,\,x \in [0,1],
\end{equation}
and there exist positive constants $C_0,c_0$, depending only on $\lambda$, $a_1$, and $A_1$, such that
$$
\|f(t) - \gamma\,\rho^{1/(r+1)}\|_{L^2([0,1])}\leq C_0\,e^{-c_0t}\qquad \text{for all}\  \,t \geq 0.
$$
\end{thm}

The result above shows the exponential convergence to equilibrium with a rate {\em independent} of the smoothness of $\rho$.
However, it does not say anything about stability of solutions.
For this, we investigate the convexity of the functional $\mathcal F_\rho$ with respect to the $2$-Wasserstein distance $W_2$.
In particular, we show that if $\|\rho'\|_\infty+\|\rho''\|_\infty$ is small enough, then $\mathcal F_\rho$ is uniformly convex.

\begin{thm}
\label{thm:2}
Let $\lambda \in (0,1]$, and assume that $\rho:[0,1]\to [\lambda,1/\lambda]$ is periodic and of class $C^{2,\alpha}$ for some $\alpha \in (0,1)$.
Let $f_1(0,\cdot),f_2(0,\cdot):[0,1]\to \mathbb R$ satisfy $0<a_1 \leq f_1(0,\cdot),f_2(0,\cdot) \leq A_1$,
and let $f_1,f_2$ solve \eqref{eq:vfd} with periodic boundary conditions, and let $a,A>0$ be as in Theorem \ref{thm:1}.
Assume that $\|\rho'\|_\infty \leq \eta_1$ and $\|\rho''\|_\infty \leq \eta_2$ for some $\eta_1,\eta_2>0$.
Then 
$$
W_2(f_1(t),f_2(t))\leq e^{-\mu\,t}W_2(f_1(0),f_2(0))\qquad \text{for all} \ \ \,t \geq 0,
$$
where
$$
\mu:=\frac{1}{A}\biggl(\frac{r(r+1) \lambda}{A^r} -\frac{2\eta_1^2(r+1)A^r}{r\lambda a^{2r} }-\frac{\eta_2}{a^r}\biggr).
$$
\end{thm}
The arguments used to prove Theorems \ref{thm:1} and \ref{thm:2} are very general, and could be applied also to the $n$-dimensional version of \eqref{eq:vfd}.
However, since the connection between this equation and the quantization problem is valid only in 1D, we have decided to state and prove these results only on the $1$
dimensional case.

\section{Maximum principle}

The goal of this section is to prove a maximum-type principle for \eqref{eq:vfd} which shows that,
if $\rho$ and $f(0)$ are bounded away from zero, then $f(t)$ remains uniformly away from zero for all $t \geq 0$.
In particular \eqref{eq:vfd} is uniformly parabolic, and $f(t)$
is smooth if $\rho$ is so. Note that, the following Proposition corresponds to the first part of Theorem \ref{thm:1}.

\begin{prop}
\label{prop:Euler}
Let $\lambda \in (0,1]$, and assume that $\rho:[0,1]\to [\lambda,1/\lambda]$ is periodic and of class $C^{k,\alpha}$ for some $k \geq 0$ and $\alpha \in (0,1)$.
Let $f(0,\cdot):[0,1]\to \mathbb R$ be a periodic function of class $C^{k,\alpha}$ satisfying $0<a_1 \leq f(0,\cdot) \leq A_1$, and let $f$ solve \eqref{eq:vfd} with periodic boundary conditions. Then
$$
 \lambda^{2/(r+1)}a_1 \leq f(t,x) \leq \frac{A_1}{\lambda^{2/(r+1)}}\qquad \text{for all} \ \,t \geq 0,
$$
$f(t,\cdot)$ is of class $C^{k,\alpha}$ for all $t \geq 0$,
and
there exists a constant $C$, depending only on
$\lambda$, $\|\rho\|_{C^{k,\alpha}}$, $k$, $\alpha$, $a_1$, and $A_1$, such that
$\|f(t,\cdot)\|_{C^{k,\alpha}([0,1])} \leq C$ for all $t \geq 0.$
\end{prop}
\begin{proof}
It is enough to prove the bound
$$
 \lambda^{2/(r+1)}a_1 \leq f(t,x) \leq \frac{A_1}{\lambda^{2/(r+1)}}\qquad \text{for all} \ \,t \geq 0,
$$
since then, once these bounds are proved, the rest of the proposition follows by standard parabolic regularity.

To prove the result, we set
$$
m(x):=\rho(x)^{1/(r+1)},\qquad u(t,x):=\frac{f(t,x)}{m(x)}.
$$
With these new unknowns \eqref{eq:vfd} becomes
\begin{equation}
\label{flow1:eq:euler comparison}
\partial_tu = -\frac{r+1}{m} \,\partial_x\biggl(m\,\partial_x\biggl(\frac{1}{u^r}\biggr) \biggr) \qquad
\text{on $[0,\infty)\times [0,1]$}
\end{equation}
with periodic boundary conditions.
The advantage of this form is that constants are solutions and we can prove a comparison principle with them.
More precisely, we set $c_0:= \lambda^{1/(r+1)}a_1$ and $C_0:=\frac{A_1}{\lambda^{1/(r+1)}}$.

Recalling the notation $s_+=\max\{s,0\}$ and $s_-=\max\{-s,0\},$ we claim that the maps
$$
t \mapsto \int_0^1 (u(t,x)-c_0)_-m\,dx \quad \text{and}\quad t \mapsto \int_0^1 (u(t,x)-C_0)_+m\,dx
$$
are nonincreasing functions.
Since $u(0,x):=\frac{f(0,x)}{m(x)}$, $a_1 \leq f(0)\leq A_1$, and $\lambda^{1/(r+1)}\leq m\leq \lambda^{-1/(r+1)}$,
it follows that
$$
\int_0^1 (u(0,x)-c_0)_-m\,dx=\int_0^1 (u(0,x)-C_0)_+m\,dx=0.
$$
Hence, thanks to the claim
$$
\int_0^1 (u(t,x)-c_0)_-m\,dx=\int_0^1 (u(t,x)-C_0)_+m\,dx=0\qquad \forall\,t \geq 0,
$$
therefore $c_0 \leq u(t,x)\leq C_0$ for all times.
Recalling that $u(t,x)=\frac{f(t,x)}{m(x)}$ and that $\lambda^{1/(r+1)}\leq m(x)\leq \lambda^{-1/(r+1)}$, this proves the result.
Hence, we only need to prove the claim.

To this aim, we only show that 
$$
t \mapsto \int_0^1 (u(t,x)-c_0)_-m\,dx 
$$
is nonincreasing (the other statement being analogous).

Since constants are solutions of \eqref{flow1:eq:euler comparison}, it holds
$$
\partial_t(u-c_0) = -\frac{r+1}{m} \,\partial_x\biggl(m\,\partial_x\biggl(\frac{1}{u^r} - \frac{1}{c_0^r}\biggr) \biggr).
$$
We now multiply the above equation by $-m\,\phi_\e\left(\frac{1}{u^r} - \frac{1}{c_0^r}\right)$,
with $\phi_\e$ a smooth approximation of the indicator function of ${\mathbb R_+}$ satisfying $\phi_\e'\geq 0$. Integrating by parts we get
\begin{align*}
\frac{d}{dt}\int_0^1 \Psi_\e(u-c_0)m\,dx&=
-\int_0^1 \phi_\e\left(\frac{1}{u^r} - \frac{1}{c_0^r}\right) \, \partial_t(u-c_0) \,m\,dx\\
&=-(r+1) \int_0^1 \biggl|\partial_x\biggl(\frac{1}{u^r} - \frac{1}{c_0^r}\biggr) \biggr|^2 \phi_\e'\biggl(\frac{1}{u^r} - \frac{1}{c_0^r}\biggr)\,m\,dx \leq 0,
\end{align*}
where we have set
$$
\Psi_\e(s):=-\int_0^s \phi_\e\biggl(\frac{1}{(\sigma+c)^r} - \frac{1}{c_0^r} \biggr)\,d\sigma.
$$
Letting $\e \to 0$ we see that $\Psi_\e(s) \to s_-$ for $s \geq -c_0$, hence
$$
\frac{d}{dt}\int_0^1 (u-c_0)_-m\,dx\leq 0,
$$
proving the result.

\end{proof}

\section{Exponential convergence to equilibrium: proof of Theorem \ref{thm:1}}
We begin by observing that,
thanks to Proposition \ref{prop:Euler}, 
$f(t)$ satisfies \eqref{eq:bound f}.
Also, recalling the definition of $\mathcal F_\rho$ (see \eqref{eq:F}), a direct computation gives
then
$$
\frac{d}{dt}\mathcal F_\rho[f(t)]=-r^2\int_0^1f(t,x)\left|\pt_x\( \frac{\rho(x)}{f(t,x)^{r+1}}\) \right|^2\,dx.
$$
Given $x \in [0,1]$, let us define the function
$$
(0,\infty)\ni s\mapsto F_x[s]:=\frac{\rho(x)}{s^r}, 
$$
so that
$$
\mathcal F_\rho[f(t)] = \int_0^1 F_x(f(t,x))\,dx.
$$
Then, 
$$
F_x[\gamma\rho(x)^{1/{r+1}}]=\frac{1}{\gamma^r}\rho(x)^{1/{r+1}},
$$
with $\gamma$ the renormalization constant of the stationary solution (so that $\gamma \rho^{1/(r+1)}$ is a probability density).

Now we will to introduce a function $G_x[f]$ that, up to translation, has the same integral of $F_x[f],$ and such that $G_x[f]$ can be used to perform an $L^2$ Gronwall estimate.
We define
$$
G_x[s]:=F_x[s]-F_x[\gamma\rho(x)^{1/{r+1}}]-F_x'[\gamma\rho(x)^{1/{r+1}}](s-\gamma\rho(x)^{1/{r+1}}).
$$
Then,
$$
G_x[s]=\frac{1}{2}\[\int_0^1F_x''[\tau s+(1-\tau)\gamma\rho(x)^{1/{r+1}}]\,d\tau\]\, (s-\gamma\rho(x)^{1/{r+1}})^2.
$$
By Proposition \ref{prop:Euler} we have that $f$ is bounded away from zero and infinity, see  \eqref{eq:bound f}.
Therefore, since $F_x$ is uniformly convex in $[a,A]$, it holds
$$
b\,|f(t,x)-\gamma\rho(x)^{1/{r+1}}|^2\le G_x[f(t,x)]\le B\,|f(t,x)-\gamma\rho(x)^{1/{r+1}}|^2 
$$
for all times, with $b,B$ positive constants.

Moreover,
$$
G_x[f(t,x)]=\frac{\rho(x)}{f(t,x)^r}-\frac{\rho(x)^{1/{r+1}}}{\gamma^r}+\frac{r}{\gamma^{r+1}}(f(t,x)-\gamma\rho(x)^{1/{r+1}}),
$$
thus, since $f$ and $\gamma\rho(x)^{1/{r+1}}$ are two probability densities, $G_x$ and $F_x$ have the same integral up to an additive constant:
$$
\int_0^1G_x[f(t,x)]\,dx=\int_0^1F_x[f(t,x)]\,dx-\int_0^1\frac{\rho(x)^{1/{r+1}}}{\gamma^r}\,dx.
$$
Therefore
\begin{align*}
\frac{d}{dt}\int_0^1G_x[f(t,x)]\,dx&=\frac{d}{dt}\int_0^1F_x[f(t,x)]\,dx\\
&=-{r^2}\int_0^1f(t,x)\left|\pt_x\( \frac{\rho(x)}{f(t,x)^{r+1}}\) \right|^2\,dx\\
&\le -{r^2}a\int_0^1\left|\pt_x\( \frac{\rho(x)}{f(t,x)^{r+1}}\) \right|^2\,dx.
\end{align*}
Notice that
$$
\pt_x\biggl(\frac{\rho(x)}{f(t,x)^{r+1}}\biggr)=\pt_x\bigg(\biggl(\frac{\rho(x)^{1/{r+1}}}{f(t,x)}\biggr)^{r+1}\biggr)=(r+1)\biggl(\frac{\rho(x)^{1/{r+1}}}{f(t,x)}\biggr)^r\pt_x\biggl(\frac{\rho(x)^{1/{r+1}}}{f(t,x)}\biggr).
$$
Thus, denoting by $c$ and $C$ positive constants depending only on $\lambda, a, A, r,$ and that $c$ and $C$ may change from line to line, we have:
\begin{align*}
\int_0^1\left|\pt_x\(\frac{\rho(x)}{f(t,x)^{r+1}}\)\right|^2\,dx&=(r+1)^2\int_0^1\bigg(\frac{\rho(x)^{1/{r+1}}}{f(t,x)}\bigg)^{2r}\bigg|\pt_x\bigg(\frac{\rho(x)^{1/{r+1}}}{f(t,x)}\bigg)\bigg|^2\,dx\\
&\ge c\int_0^1\bigg|\pt_x\bigg(\frac{\rho(x)^{1/{r+1}}}{f(t,x)}\bigg)\bigg|^2\,dx\\
&\ge c\int_0^1\bigg|\frac{\rho(x)^{1/{r+1}}}{f(t,x)}-\frac{1}{\alpha(t)}\bigg|^2\,dx\\
&=c \int_0^1\frac{1}{\alpha^2(t) f(t,x)^2}\bigg|\alpha(t)\rho^{1/{r+1}}-f(t,x) \bigg|^2\,dx,
\end{align*}
where
$$
\alpha(t)=\int_0^1\frac{\rho(x)^{1/{r+1}}}{f(t,x)}\,dx
$$
is bounded away from zero and infinity for all times (thanks to  \eqref{eq:bound f} and the bound $\lambda \leq \rho\leq 1/\lambda$):
$$
0<c\le\alpha(t)\le C<\infty.
$$
Therefore
\begin{align*}
\int_0^1\left|\pt_x\(\frac{\rho(x)}{f(t,x)^{r+1}}\)\right|^2\,dx&\ge c \int_0^1\frac{1}{\alpha^2(t) f(t,x)^2}\left|\alpha(t)\rho^{1/{r+1}}-f(t,x) \right|^2\,dx\\
&\ge c  \int_0^1\left|\alpha(t)\rho^{1/{r+1}}-f(t,x) \right|^2\,dx.\\
\end{align*}
Now, the problem is that $\alpha(t)$ \emph{a priori} does not coincide with $\gamma.$
For this reason we use the following trick: 
\begin{align*}
\int_0^1\left|\pt_x\(\frac{\rho(x)}{f(t,x)^{r+1}}\)\right|^2\,dx&\ge c  \int_0^1\left|\alpha(t)\rho^{1/{r+1}}-f(t,x) \right|^2\,dx\\
&=c \int_0^1\rho(x)^{1/{r+1}}\left|\alpha(t)\rho^{1/2({r+1})}-\frac{f(t,x)}{\rho(x)^{1/{2(r+1)}}} \right|^2\,dx\\
&\ge c \int_0^1\left|\alpha(t)\rho^{1/2(r+1)}-\frac{f(t,x)}{\rho(x)^{1/{2(r+1)}}} \right|^2\,dx\\
&\ge c \min_{\beta} \int_0^1\left|\beta\rho^{1/2(r+1)}-\frac{f(t,x)}{\rho(x)^{1/{2(r+1)}}} \right|^2\,dx\\
&= c \int_0^1\left|\gamma\rho^{1/2(r+1)}-\frac{f(t,x)}{\rho(x)^{1/{2(r+1)}}} \right|^2\,dx\\
&= c \int_0^1\left|\gamma\rho^{1/{r+1}}-f(t,x)\right|^2 \frac{1}{\rho(x)^{1/{r+1}}}\,dx\\
&\ge c \int_0^1\left|\gamma\rho^{1/{r+1}}-f(t,x)\right|^2\,dx\\
&\ge \frac{c}{B}\int_0^1 G_x[f](t,x)dx.
\end{align*}
Therefore, by Gronwall Lemma, there exists a constant $\hat{c}$ such that 
$$
\int_0^1 G_x[f(t,x)]\,dx\le e^{-\hat{c}t}\int_0^1 G_x[f(0,x)]\,dx.
$$
Since $G_x[f(t,x)]$ is comparable to $\left|f(t,x)-\gamma \rho^{1/{r+1}}\right|^2$, this Gronwall estimate implies the exponential convergence of $f$ to the stationary solution $\gamma \rho^{1/{r+1}}$, namely
$$
\int_0^1\left|f(t,x)-\gamma \rho^{1/{r+1}}\right|^2\le \hat C\, e^{-\hat{c}t},
$$
as desired

\section{Stability in $W_2$: proof of Theorem \ref{thm:2}}
To prove Theorem \ref{thm:2}, we shall first compute the Hessian of $\mathcal F_\rho[f]$ at a fixed probability density $f$,
and then we apply this estimate to prove the contraction along two solutions of \eqref{eq:vfd}.
Since, under our assumptions, solutions are of class $C^{2,\alpha}$, in the next section we assume that $f \in C^2$.

\subsection{Hessian of $\mc{F}_{\rho}[f]$}
In this section we compute the Hessian of 
$$
\mathcal F_\rho[f]=\int_0^1\frac{\rho(x)}{f(x)^{r}} \,dx
$$
with respect to $W_2$. For this, we use the Riemannian formalism introduced in \cite{O}. 

Our state space $\mathcal{M}$ is the space of positive functions $f:(0,1)\to (0, \infty)$ with unit integral:
$$
\int_0^1f\,dx=1.
$$
We may think of infinitesimal perturbations $\delta f \in T_f\mathcal M$ of a state $f\in \mathcal M$ as functions $\delta f: (0,1) \to \br$ with
\be\label{eq:int0}
\int_0^1\delta f\,dx=0.
\ee
For given $f\in \mathcal M$ we define the scalar product $g_f$ on $T_f\mathcal M$ via
$$
g_f(\delta f_0, \delta f_1):= \int_0^1 \pt_x \phi_0 \,\pt_x \phi_1 fdx, 
$$
where, up to additive constants, the functions $\phi_i: (0,1)\to \br$ are definite by
\be\label{delta f}
\delta f_i-\pt_x (f \pt_x\phi_i)=0.
\ee
Note that, since the variational derivative of $\mathcal F_\rho[f]$ is given by
$$
\frac{\delta \mathcal F_\rho[f]}{\delta f}= -r\frac{\rho(x)}{f(x)^{r+1}},
$$
the equation \eqref{eq:vfd} can be interpreted as the gradient flow of the functional $\mc{F}_{\rho}[f]$ in the $2$-Wasserstein metric:
\be
\label{eq:eulerian GF}
\pt_tf(t,x)=-\operatorname{grad}_W \mathcal F_\rho[f(t)]= \pt_x\bigg(f(t,x)\pt_x\(\frac{\delta \mathcal F_\rho[f(t)]}{\delta f}\) \bigg).
\ee

Now, given a periodic probability density $f:[0,1]\to (0,\infty)$ of class $C^2$, let the function $\delta f$ satisfy \eqref{eq:int0}, and let $\phi$ be related to $\delta f$ by \eqref{delta f}.

We compute the first derivative of $\mathcal F_\rho[f].$ Using that
$$
\frac{\pt_x f}{f^{r+1}}=-\frac{1}{r}\pt_x\(\frac{1}{f^r}\),
$$
we have:
\begin{align*}
\La\frac{\delta \mathcal F_\rho[f]}{ \delta f},\delta f\Ra&=-r\int_0^1 \frac{\rho}{f^{r+1}}\,\delta f\,dx\overset{\eqref{delta f}}{=}-r\int_0^1 \frac{\rho}{f^{r+1}}\pt_x (f \pt_x\phi)\,dx\\
&=-r\int_0^1 \frac{\rho}{f^r}\pt_{xx}\phi\,dx-r\int_0^1 \frac{\rho}{f^{r+1}}\pt_x f\, \pt_{x}\phi\,dx\\
&=-r\int_0^1 \frac{\rho}{f^r}\pt_{xx}\phi\,dx+\int_0^1 \rho\, \pt_x\(\frac{1}{f^{r}}\)\pt_{x}\phi\,dx\\
%=-2\int_0^1 \frac{\rho(x)}{f(t,x)^{2}}\pt_{xx}\phi\,dx-\int_0^1 \frac{\pt_x\rho(x)}{f(t,x)^{2}}\pt_{x}\phi\,dx-\int_0^1\frac{\rho(x)}{f(t,x)^{2}}\pt_{xx}\phi\,dx\\
&=-\int_0^1 \frac{\pt_x\rho}{f^r}\pt_{x}\phi\,dx-(r+1)\int_0^1 \frac{\rho}{f^r}\pt_{xx}\phi\,dx.
\end{align*}
Now, to compute the Hessian of $\mathcal F_\rho$, we consider a geodesic $f:[0,1]\to \mathcal M$ such that $f(0)=f$.
Then the Hessian of $\mathcal F_\rho$ at $f$ is computed by considering 
$$
\frac{d^2}{d^2s}\Big|_{s=0}\mathcal F_\rho[f(s)].
$$
Recall that the geodesic equation is given by the system
\be\label{eq: var f}
\pt_s f-\pt_x (f \pt_x\phi)=0
\ee
\be\label{eq:HJ}
\pt_s \phi-\frac{1}{2}|\pt_x\phi|^2=0,
\ee
(see for instance \cite[Sections 2 and 3.2]{OW})
and that, with this notation,
\begin{align*}
\frac{d}{ds}\mathcal F_\rho[f(s)]&=\La\frac{\delta \mathcal F_\rho[f(s)]}{ \delta f}, \delta f(s)\Ra\\
&=-\int_0^1 \frac{\pt_x\rho}{f^r}\pt_{x}\phi\,dx-(r+1)\int_0^1 \frac{\rho}{f^r}\pt_{xx}\phi\,dx,
\end{align*}
where $\delta f(s)$ is related to $\phi(s)$ by \eqref{delta f}.

We now compute the second derivative of $\mathcal F_\rho[f(s)]$:
\begin{align*}
\frac{d^2}{d^2s}\Big|_{s=0}\mathcal F_\rho[f(s)]&=\frac{d}{ds}\biggl(-\int_0^1 \frac{\pt_x\rho}{f^r}\pt_{x}\phi\,dx-(r+1)\int_0^1 \frac{\rho}{f^r}\pt_{xx}\phi\,dx,\biggr)\\
&=r\int_0^1 \frac{\pt_x\rho}{f^{r+1}}\pt_sf\,\pt_x \phi\,dx-\int_0^1 \frac{\pt_x\rho}{f^r}\pt_x (\pt_s\phi )\,dx\\
&+ r(r+1) \int_0^1\frac{\rho}{f^{r+1}}\pt_sf\,\pt_{xx} \phi\,dx-(r+1)\int_0^1 \frac{\rho}{f^r}\pt_{xx} (\pt_s\phi )\,dx\\
&\overset{\eqref{eq: var f}+\eqref{eq:HJ}}{=}r\int_0^1 \frac{\pt_x\rho}{f^{r+1}}\pt_x (f \pt_x\phi)\pt_x \phi\,dx-\int_0^1\frac{\pt_x\rho}{f^r}\pt_x\(\frac{1}{2}|\pt_x\phi|^2 \)\,dx\\
&+ r(r+1) \int_0^1\frac{\rho}{f^{r+1}}\pt_x (f \pt_x\phi)\pt_{xx} \phi\,dx-(r+1)\int_0^1 \frac{\rho}{f^r}\pt_{xx}\(\frac{1}{2}|\pt_x\phi|^2 \)\,dx\\
&=r\int_0^1 \frac{\pt_x\rho}{f^{r+1}}\pt_x f (\pt_x \phi)^2\,dx+r\int_0^1 \frac{\pt_x\rho}{f^{r}}\pt_x\phi\,\pt_{xx} \phi\,dx\\
&-\int_0^1\frac{\pt_x\rho}{f^r}\pt_x\(\frac{1}{2}|\pt_x\phi|^2 \)\,dx+ r(r+1) \int_0^1\frac{\rho}{f^{r}}(\pt_{xx} \phi)^2\,dx\\
&+ r(r+1) \int_0^1\frac{\rho}{f^{r+1}}\pt_xf\,\pt_{x} \phi\,\pt_{xx} \phi\,dx-(r+1)\int_0^1 \frac{\rho}{f^r}\pt_{xx}\(\frac{1}{2}|\pt_x\phi|^2 \)\,dx.
%=-2\int_0^1 \frac{\rho(x)}{f(t,x)^{2}}\pt_{xx}\phi\,dx-2\int_0^1 \frac{\rho(x)}{f(t,x)^{3}}\pt_x f(t,x)\pt_{x}\phi\,dx\\
%=-2\int_0^1 \frac{\rho(x)}{f(t,x)^{2}}\pt_{xx}\phi\,dx+\int_0^1 \rho(x)\pt_x\(\frac{1}{f(t,x)^{2}}\)\pt_{x}\phi\,dx\\
%=-2\int_0^1 \frac{\rho(x)}{f(t,x)^{2}}\pt_{xx}\phi\,dx-\int_0^1 \frac{\pt_x\rho(x)}{f(t,x)^{2}}\pt_{x}\phi\,dx-\int_0^1\frac{\rho(x)}{f(t,x)^{2}}\pt_{xx}\phi\,dx\\
%=-\int_0^1 \frac{\pt_x\rho(x)}{f(t,x)^{2}}\pt_{x}\phi\,dx-3\int_0^1 \frac{\rho(x)}{f(t,x)^{2}}\pt_{xx}\phi\,dx.
\end{align*}
Using again that
$\frac{\pt_x f}{f^{r+1}}=-\frac{1}{r}\pt_x\(\frac{1}{f^r}\),$ and integrating by parts,
we get
\begin{align*}
\frac{d^2}{d^2s}\Big|_{s=0}\mathcal F_\rho[f(s)]&=
%=+2\int_0^1 \frac{\pt_x\rho}{f^{3}}\pt_x f (\pt_x \phi)^2\,dx+2\int_0^1 \frac{\pt_x\rho}{f^{2}}\pt_x\phi\,\pt_{xx} \phi\,dx-\int_0^1\frac{\pt_x\rho(x)}{f(t,x)^{2}}\pt_x\(\frac{1}{2}|\pt_x\phi|^2 \)\,dx\\
%+ 6 \int_0^1\frac{\rho}{f^{2}}(\pt_{xx} \phi)^2\,dx+ 6 \int_0^1\frac{\rho}{f^{3}}\pt_sf\,\pt_{x} \phi\,\pt_{xx} \phi\,dx-3\int_0^1 \frac{\rho(x)}{f(t,x)^{2}}\pt_{xx}\(\frac{1}{2}|\pt_x\phi|^2 \)\,dx\\
-\int_0^1 \pt_x\rho\,\pt_x\(\frac{1}{f^r}\)(\pt_x \phi)^2\,dx+r\int_0^1 \frac{\pt_x\rho}{f^{r}}\pt_x\phi\,\pt_{xx} \phi\,dx\\
&-\int_0^1\frac{\pt_x\rho}{f^r}\pt_x\(\frac{1}{2}|\pt_x\phi|^2 \)\,dx+ r(r+1) \int_0^1\frac{\rho}{f^{r}}(\pt_{xx} \phi)^2\,dx\\
&-(r+1) \int_0^1\rho\,\pt_x\(\frac{1}{f^r}\)\pt_x\phi\,\pt_{xx} \phi\,dx-(r+1)\int_0^1 \frac{\rho}{f^r}\pt_{xx}\(\frac{1}{2}|\pt_x\phi|^2 \)\,dx\\
&=\int_0^1 \frac{\pt_{xx}\rho}{f^r}(\pt_{x} \phi)^2\,dx+\int_0^1 \frac{\pt_{x}\rho}{f^r}\pt_x\big((\pt_{x} \phi)^2\big)\,dx\\
&+r\int_0^1 \frac{\pt_x\rho}{f^{r}}\pt_x\phi\,\pt_{xx} \phi\,dx-\int_0^1\frac{\pt_x\rho}{f^r}\pt_x\(\frac{1}{2}|\pt_x\phi|^2 \)\,dx\\
&+ r(r+1) \int_0^1\frac{\rho}{f^{r}}(\pt_{xx} \phi)^2\,dx
+(r+1)\int_0^1 \frac{\pt_{x}\rho}{f^r}\pt_x \phi\,\pt_{xx}\phi\,dx\\
&+(r+1) \int_0^1\frac{\rho}{f^r}\pt_x\(\pt_x\phi\,\pt_{xx} \phi\)\,dx
-(r+1)\int_0^1 \frac{\rho}{f^r}\pt_{xx}\(\frac{1}{2}|\pt_x\phi|^2 \)\,dx\\
&=2(r+1)\int_0^1 \frac{\pt_{x}\rho}{f^r}\pt_x \phi\,\pt_{xx}\phi\,dx+(r+1)^2 \int_0^1\frac{\rho}{f^{r}}(\pt_{xx} \phi)^2\,dx\\
&+\int_0^1 \frac{\pt_{xx}\rho}{f^r}(\pt_{x} \phi)^2\,dx-(r+1)\int_0^1 \frac{\rho}{f^{r}}\Bigg[-\pt_x\phi\, \pt_{xxx}\phi+\pt_{xx}\(\frac{1}{2}|\pt_x\phi|^2 \) \Bigg]\,dx.
\end{align*}
We now notice that
$$
-\pt_x\phi \,\pt_{xxx}\phi+\pt_{xx}\(\frac{1}{2}|\pt_x\phi|^2 \)=(\pt_{xx} \phi)^2,
$$
so we get
\begin{multline*}
\frac{d^2}{d^2s}\Big|_{s=0}\mathcal F_\rho[f(s)]=
2(r+1)\int_0^1 \frac{\pt_{x}\rho}{f^r}\pt_x \phi\pt_{xx}\phi\,dx\\
+r(r+1) \int_0^1\frac{\rho}{f^{r}}(\pt_{xx} \phi)^2\,dx
+\int_0^1 \frac{\pt_{xx}\rho}{f^r}(\pt_{x} \phi)^2\,dx.
\end{multline*}
We now want to investigate the $\mu$-convexity of the functional $\mathcal F_\rho$
in terms of the assumptions on $\rho$ and $f$.

Assume that $\rho$ is a periodic probability density of class $C^{2,\alpha}$ with $\|\rho'\|_\infty \leq \eta_1,$ and $\|\rho''\|_\infty \leq \eta_2.$ We assume also that $0<\lambda \leq \rho \leq 1/\lambda$,
and that $0<a\leq f \leq A$.
Then
\begin{multline*}
\frac{d^2}{d^2s}\Big|_{s=0}\mathcal F_\rho[f(s)]\ge
-\frac{2\eta_1(r+1)}{a^r}\int_0^1 |\pt_x\phi| |\pt_{xx}\phi|\,dx\\
+\frac{r(r+1) \lambda}{A^r} \int_0^1(\pt_{xx} \phi)^2\,dx
-\frac{\eta_2}{a^r}\int_0^1 (\pt_{x} \phi)^2\,dx.
\end{multline*}
By Young inequality we have, for any $\e>0$,
\begin{multline*}
\frac{d^2}{d^2s}\Big|_{s=0}\mathcal F_\rho[f(s)]\ge
-\frac{\eta_1(r+1)}{\e a^r}\int_0^1 |\pt_x\phi|^2\,dx -\frac{\e \eta_1(r+1)}{a^r}\int_0^1 |\pt_{xx}\phi|^2\,dx\\
+\frac{r(r+1) \lambda}{A^r} \int_0^1(\pt_{xx} \phi)^2\,dx-\frac{\eta_2}{a^r}\int_0^1(\pt_{x} \phi)^2\,dx.
\end{multline*}
Choosing $\e=\frac{r\lambda a^r}{2\eta_1 A^r}$, we get
\begin{multline*}
\frac{d^2}{d^2s}\Big|_{s=0}\mathcal F_\rho[f(s)]\ge
\frac{r(r+1) \lambda}{2A^r}\int_0^1|\pt_{xx}\phi|^2\,dx
-\biggl(\frac{2\eta_1^2(r+1)A^r}{r\lambda a^{2r} }+\frac{\eta_2}{a^r} \biggr) \int_0^1 |\pt_x\phi|^2\,dx
\end{multline*}
Using Poincar\'e inequality on $[0,1]$ (recalling that the Poincar\'e constant is $1/2$), we obtain
\begin{align*}
\frac{d^2}{d^2s}\Big|_{s=0}\mathcal F_\rho[f(s)]&\ge
\biggl( \frac{r(r+1) \lambda}{A^r} -\frac{2\eta_1^2(r+1)A^r}{r\lambda a^{2r} }-\frac{\eta_2}{a^r}\biggr) \frac{1}{A}\int_0^1f|\pt_{x}\phi|^2\,dx\\
&\geq \mu\int_0^1f|\pt_{x}\phi|^2\,dx,
\end{align*}
where
$$
\mu:=\frac{1}{A}\biggl(\frac{r(r+1) \lambda}{A^r} -\frac{2\eta_1^2(r+1)A^r}{r\lambda a^{2r} }-\frac{\eta_2}{a^r}\biggr).
$$
This proves that the Hessian of $\mathcal F_\rho$ at $f$ is bounded from below by $\mu$.

\subsection{Application to stability of solutions to \eqref{eq:vfd}}
As we shall explain in the next section,
to ensure that the above convexity results can be applied to equation \eqref{eq:vfd}, one needs to know that if $f_1(t,x), f_2(t,x)$
are solutions of \eqref{eq:vfd}, and if 
$$
[0,1]\ni s\mapsto f^s(t,x)
$$
is a Wasserstein geodesic such that $f^0(t,x)=f_1(t,x)$ and $f^1(t,x)=f_2(t,x)$, then there exist constants $a,A>0$ such that
$$
0<a\leq f^s(t,x) \leq A\qquad \forall\,s \in [0,1],\, \forall\,t\geq 0,\, \forall\, x.
$$
Thanks to Theorem \ref{thm:1}, we know that the above bounds hold at $s=0,1$, for all $t,x$.

We now fix $t \geq 0$ and consider $s\mapsto f^s$ the geodesic connecting $ f_1(t)$ to $f_2(t)$ on
$(\mc M, W_2).$ 

The goal is to show that
$$
\Hess_{W_2}\mc F_\r[f^s] \geq \mu \quad \mbox{for all}\ s\in[0,1],
$$
and as explained above,
to prove this result it is enough to prove the following implication:
\be
\label{eq:bounds per f^s}
a\leq f_1(t),f_2(t) \leq  A\quad
 \Longrightarrow   \quad a \leq f^s \leq  A\quad \mbox{for all} \ s \in [0,1].
\ee
Let $T$ be the optimal transport map from $f_1(t)$ to $f_2(t).$ By definition $f^s$ is given by
$$
(T_s)_\# f_1(t) = f^s \quad  \mbox{where}\ \  T_s(x)=(1-s)x+s T(x).
$$   
By definition of push-forward we have
\be
\label{eq:PF T_s}
T_s' = \frac{f_1(t)}{f_s\circ T_s}
\ee
and
\be
\label{eq:PF T}
T'=\frac{f_1(t)}{f_2(t)\circ T }.
\ee
Let us prove \eqref{eq:bounds per f^s}. By \eqref{eq:PF T_s} and \eqref{eq:PF T} we have:
\begin{multline*}
f^s\circ T_s=\frac{f_1(t)}{T_s'}=\frac{f_1(t)}{sT'+(1-s)}\\
=\frac{f_1}{\frac{sf_1(t)+(1-s)f_2(t)\circ T}{f_2(t)\circ T} }=\frac{f_1(t)\,f_2(t)\circ T}{sf_1(t)+(1-s)f_2(t)\circ T }.
\end{multline*}
Noticing that
$$
\min\{f_1(t);f_2(t)\circ T\}\le \frac{f_1(t)\,f_2(t)\circ T}{sf_1(t)+(1-s)f_2(t)\circ T }\le \max\{f_1(t);f_2(t)\circ T\}
$$
we obtain the validity of \eqref{eq:bounds per f^s}.
In the next subsection,
we briefly summarize the general consequences of $\mu$-convexity and we conclude the proof of Theorem \ref{thm:2}.

\subsection{$W_2$-stability}
In this section we use Otto's formalism to deduce convergence and stability of solutions.
Although these computations are formal, we present them as they show in a very elegant way why convexity of $\mathcal F$
implies such stability.
For a rigorous proof, the reader may look at the paper \cite[Section 4]{OW}.

Recall that, formally, our equation \eqref{eq:vfd} can be written as 
\be
\label{eq:GFpde}
\dot f=-\nabla_{W_2}\mc F_\r[f],
\ee
where
$$
\nabla_{W_2}\mc F_\r[f]=r\Div\(f\nabla \(\frac{\rho}{f^{r+1}}\)\).
$$
Now, given two solutions $f_1$ and $f_2$ as in the statement of the theorem, 
and denoting by $f^s$ the geodesic connecting them, we compute
\begin{align*}
\frac{d}{dt}\frac{W_2(f_1,f_2)^2}{2} &= g_{f_1}\Big(\dot f_1, \pt_s f^s\big|_{s=0} \Big)
-g_{f_2}\Big(\dot f_2, \pt_s f^s\big|_{s=1} \Big)\\
&=-g_{f_1}\Big(\nabla_{W_2}\mc F_\r[f_1], \pt_s f^s\big|_{s=0} \Big)
+g_{f_2}\Big(\nabla_{W_2}\mc F_\r[f_2], \pt_s f^s\big|_{s=1} \Big)
\end{align*}
Now, since $f^s$ is a geodesic,
$$
\frac{d}{ds}  g_{f^s}\Big(\nabla_{W_2}\mc F_\r[f^s], \pt_s f^s \Big)=g_{f^s}\bigg(\Hess_{W_2}\mc F_\r [f^s]\,\pt_s f^s,\pt_s f^s\bigg).
$$
Thus
\begin{align*}
-g_{f_1}\Big(\nabla_{W_2}\mc F_\r[f_1], \pt_s f^s\big|_{s=0} \Big)&
+g_{f_2}\Big(\nabla_{W_2}\mc F_\r[f_2], \pt_s f^s\big|_{s=1} \Big)\\
&=-\int_0^1g_{f^s}\bigg(\Hess_{W_2}\mc F_\r [f^s]\,\pt_s f^s,\pt_s f^s\bigg)\,ds\\
&\leq -\mu \int_0^1g_{f^s}(\pt_s f^s,\pt_s f^s)\,ds=-\mu\,W_2(f_1,f_2)^2,
\end{align*}
where in the last inequality we used again that $f^s$ is a geodesic.
Hence, combining these two equations we get
$$
\frac{d}{dt}\frac{W_2(f_1,f_2)^2}{2} \leq -2\mu \,\frac{W_2(f_1,f_2)^2}{2},
$$
which gives the result.

\bigskip

{\it Acknowledgments:}   The author is grateful to Prof. Jos\'e Antonio Carrillo and to Francesco Saverio Patacchini for useful comments and for carefully reading this manuscript; and to Prof. Matteo Bonforte for suggesting references on very fast diffusion equations.

The author would also like to acknowledge the L'Or\'eal Foundation for partially supporting this project by awarding the L'Or\'eal-UNESCO \emph{For Women in Science France fellowship}.

\end{document}